\documentclass[a4,12pt,leqno]{amsart}
\usepackage{amsmath,amssymb,latexsym,cite}
\usepackage[small]{caption}
\usepackage{graphicx,color,mathrsfs,tikz}
\usepackage{subfigure,color}
\usepackage{cite}
\usepackage[colorlinks=true,urlcolor=blue,
citecolor=red,linkcolor=blue,linktocpage,pdfpagelabels,
bookmarksnumbered,bookmarksopen]{hyperref}
\usepackage[english]{babel}
\usepackage{units}
\usepackage{enumitem}
\usepackage[left=2.6cm,right=2.6cm,top=2.9cm,bottom=2.9cm]{geometry}
\usepackage[hyperpageref]{backref}

\usepackage[colorinlistoftodos,prependcaption]{todonotes}
\usepackage{xcolor}

\numberwithin{equation}{section}
\newtheorem{theorem}{Theorem}[section]
\newtheorem*{theoMPSC}{Theorem}
\newtheorem{proposition}[theorem]{Proposition}
\newtheorem{lemma}[theorem]{Lemma}
\newtheorem{remark}[theorem]{Remark}

\newtheorem{corollary}[theorem]{Corollary}
\newtheorem{definition}[theorem]{Definition}
\theoremstyle{definition}

\renewcommand{\epsilon}{\eps}

\newcommand{\N}{{\mathbb N}}
\newcommand{\R}{{\mathbb R}}

\newcommand{\G}{{\mathbb{G}}}

\newcommand{\eps}{\varepsilon}

\newcommand{\pnorm}[2][]{\if #1'' \left|#2\right|_p \else \left|#2\right|_{#1} \fi}

\newcommand{\scal}[2]{\langle {#1} , {#2}\rangle}

\renewcommand{\theta}{\vartheta}

\makeatletter
\def\cleardoublepage{\clearpage\if@twoside \ifodd\c@page\else
\hbox{}
\thispagestyle{empty}
\newpage
\if@twocolumn\hbox{}\newpage\fi\fi\fi}
\makeatother
\title{Integral representation of local left--invariant functionals in Carnot groups}
\author{A. Maione}
\address{Alberto Maione: Dipartimento di Matematica\\Universit\`a di Trento\\ Via Sommarive 14\\ 38123, Povo (Trento) - Italy\\}
\email{alberto.maione@unitn.it}
\thanks{A.M. is supported by MIUR, Italy, GNAMPA of INDAM and University of Trento, Italy.}
\date{\today}
\author{E. Vecchi}
\address{Eugenio Vecchi: Dipartimento di Matematica\\Politecnico di Milano\\ Via Edoardo Bonardi 13\\ 20133, Milano - Italy\\}
\email{eugenio.vecchi@polimi.it}


\begin{document}

\begin{abstract}
The aim of this note is to prove a representation theorem for left--invariant
functionals in Carnot groups. As a direct consequence, we can also provide
a $\Gamma$-convergence result for a smaller class of functionals.
\end{abstract}


\maketitle

\section{Introduction}
The representation of local functionals as integral functionals has a very long history and
exhibits a natural application when dealing with relaxed functionals and $\Gamma$-limits
in a suitable topology. 
In the Euclidean case this problem is very well understood and we refer the reader to the
papers \cite{BDM,BDM2,Alberti} as well as the classical monographs \cite{DM,B,Bra} and the references therein.\\
The same problem may be faced when dealing with abstract functionals defined on Sobolev spaces
built starting from a family of vector fields satisfying certain natural conditions. 
This is the starting point of the recent paper \cite{MPSC} 
where the authors started the study of very general functionals proving, among many other results,
that they can be represented as integral functionals whose integrand depends
on a gradient modeled on a family of vector fields. In order to better understand
the motivation behind our work, let us be more specific about one of the results
proved in \cite{MPSC}. For the sake of simplicity, we state it in the more specific
context of Carnot groups. We refer to Section \ref{section3} for a detailed account of 
all the definitions needed in the following
\begin{theoMPSC}[\cite{MPSC},\textrm{Theorem} 3.12]
Let $\Omega\subset\mathbb{G}$ be a bounded open set.
Let $p \in (1,+\infty)$ and let $\mathcal{A}$ the class of all open subsets of $\Omega$. 
Let $F:L^p(\Omega)\times\mathcal{A}\to [0,+\infty]$ 
be an increasing functional satisfying the following properties:
\begin{itemize}
	\item $F$ is local;
	\item $F$ is a measure;
	\item $F$ is convex and lower semicontinuous;
	\item $F(u+c,A)=\,F(u,A)$ for each $u\in L^p(\Omega)$, $A\in\mathcal{A}$ and $c\in\mathbb{R}$;
	\item there exist a nonnegative function $a \in L^{1}_{\mathrm{loc}}(\Omega)$ 
	and a positive constant $b\in\mathbb{R}$ such that
	\[
	0\le\,F(u,A)\le\,\int_A\left(a(x)+b\,|\nabla_{\G}u(x)|^p\right)\,dx
	\]
	for each $u\in C^{1}(A)$, $A\in\mathcal{A}$.
\end{itemize}
Then there exists a Borel function $f:\mathbb{R}^m\to [0,+\infty]$ such that
\begin{itemize}
  \item for a.e. $x \in \Omega$, $f(x,\cdot)$ is convex;
	\item  for each $u\in L^p(\Omega)$, for each $A\in\mathcal{A}$ with $u|_A\in W^{1,p}_{\G, \mathrm{loc}}(A)$ we have
	\[
	F(u,A)=\int_A f(x,\nabla_{\G}u(x))\,dx\,;
	\]
	\item for a.e. $x\in \Omega$, 
	$$0\le f(\eta)\le a(x)+\, b \, |\eta|^p\ \forall\,\eta\in\mathbb{R}^m.$$
\end{itemize}
\end{theoMPSC} 
We stress that the above result actually hold for a far more general class of vector fields,
not necessarily related to a Carnot group structure, see \cite{MPSC, MPSC2} for more details.

\medskip
The aim of this note is twofold: on the one hand, we are interested in proving that
under the extra condition of being {\it left--invariant}, see Definition \ref{DefLIF}, 
the functional is obviously still represented
by an integral, {\it but} the integrand does not depend 
anymore on the point, but only on the (intrinsic) gradient. \\
The second goal is to prove that the {\it left--invariant} condition allows to represent the functional on a wider class
of functions, namely $W^{1,1}_{\G,\mathrm{loc}}$, and not only on $W^{1,p}_{\G,\mathrm{loc}}$ for $p>1$, 
see Definition \ref{DefSobolevSpace}.
This is actually not really a surprise, because this is precisely
what happens in the classical (Euclidean) case 
when dealing with {\it translation invariant functionals}, see e.g. \cite[Chapter 23]{DM}. 
Nevertheless, the above mentioned results cannot be directly applied in our case, indeed it is not
difficult to produce examples of functionals which are left--invariant 
(w.r.t a Carnot group structure) but not translation invariant
in the Euclidean sense.
We also want to stress that one of the key ingredients to get the representation over $W^{1,1}_{\G,\mathrm{loc}}$
is provided by the use of the local convolution. This tool is far more delicate in
the context of Carnot groups and it has been recently introduced and deeply studied in \cite{CM}.\\

We are now ready to state the main result of this paper, which
is a representation theorem for local left--invariant functionals.
\begin{theorem}\label{Th1}
Let $p \in [1,+\infty)$ and let $\mathcal{A}_0$ the class of all bounded open subsets of $\G$. Let $F:L_{\mathrm{loc}}^p(\G)\times\mathcal{A}_0\to [0,+\infty]$ be an increasing functional satisfying the following properties:
\begin{itemize}
	\item[(a)] $F$ is local and left-invariant;
	\item[(b)] $F$ is a measure;
	\item[(c)] $F$ is convex and lower semicontinuous;
	\item[(d)] $F(u+c,A)=\,F(u,A)$ for each $u\in L_{\mathrm{loc}}^p(\G)$, $A\in\mathcal{A}_0$ and $c\in\mathbb{R}$;
	\item[(e)] there exist $a,b\in\mathbb{R}^+$ such that
	\[
	0\le\,F(u,A)\le\,\int_A\left(a+b\,|\nabla_{\G}u(x)|^p\right)\,dx
	\]
	for each $u\in W_{\G, \mathrm{loc}}^{1,1}(A)$, $A\in\mathcal{A}_0$.
\end{itemize}
Then there exists a convex function $f:\mathbb{R}^m\to [0,+\infty]$ such that
\begin{itemize}
	\item[(i)] for each $u\in L_{\mathrm{loc}}^p(\G)$, 
	for each $A\in\mathcal{A}_0$ with $u|_A\in W^{1,1}_{\G, \mathrm{loc}}(A)$ we have
	\[
	F(u,A)=\int_A f(\nabla_{\G}u(x))\,dx\,;
	\]
	\item[(ii)]	$$0\le f(\eta)\le a+\, b\, |\eta|^p\ \forall\,\eta\in\mathbb{R}^m.$$
\end{itemize}
\end{theorem}
\noindent We refer once again to Section \ref{section3} for details.\\
We believe that Theorem \ref{Th1} has its own interest, nevertheless we can immediately
apply it, in combination with other results proved in \cite{MPSC}, to get
a $\Gamma$--compactness result for left--invariant
functionals. Essentially, this says that, up to subsequences, 
the $\Gamma$-limit of a sequence of left--invariant functionals
exists and it is a left--invariant functional as well.
We refer to the nowadays classical texts \cite{Bra,DM} for an introduction to $\Gamma$-convergence.

\begin{corollary}\label{Th2}
Let $p \in (1,+\infty)$ and let $\mathcal{A}_0$ be the class of all open bounded subsets of $\G$. 
Then, for every sequence $\{F_h\}_{h\in\mathbb{N}}$ of increasing functionals $F_h:L_{\mathrm{loc}}^p(\G)\times\mathcal{A}_0 \to [0,+\infty]$ satisfying hypothesis $(a)-(e)$ of Theorem \ref{Th1} for every $h\in\mathbb{N}$, up to a subsequence, there exists a local left--invariant functional 
$F:L_{\mathrm{loc}}^p(\G)\times\mathcal{A}_0 \to [0,+\infty]$ such that 
\begin{equation*}
F(\cdot,A)=\Gamma-\lim_{h\to +\infty} F_h(\cdot,A)\quad\text{for each }A\in\mathcal{A}_0
\end{equation*}
\noindent in the $L^{p}_{\mathrm{loc}}(\G)$ topology.
Moreover, there exists a convex function $f:\mathbb{R}^m\to[0,+\infty]$ such that
\begin{itemize}
	\item[(i)] for each $u\in L_{\mathrm{loc}}^p(\G)$, 
	for each $A\in\mathcal{A}_0$ with $u|_A\in W^{1,p}_{\G, \mathrm{loc}}(A)$ we have
	\[
	F(u,A)=\int_A f(\nabla_{\G}u(x))\,dx\,;
	\]
	\item[(ii)]	$$0\le f(\eta)\le a+\, b\, |\eta|^p\ \forall\,\eta\in\mathbb{R}^m.$$
\end{itemize}
\end{corollary}

\medskip

A natural comparison with \cite{MPSC} is now in order. As already mentioned, the representation theorem 
proved in \cite{MPSC} does not require any Carnot group structure. 
Nevertheless, this setting seems to be quite necessary in 
order to be able to speak about some form of {\it invariance}, in this case with respect to
the group law. It would be interesting to study similar results in more general contexts, introducing
appropriate notions of invariance.\\
We also want to stress that several results concerning homogenization and $H$-convergence of operators
in Carnot groups are already available in the literature, see e.g. \cite{DDMM,FranGut,FranTesi,M}.\\

\medskip
\indent The structure of the paper is the following: in Section \ref{Prel} we provide
the basic necessary notions of Carnot groups and of local convolution within
Carnot groups. In Section \ref{section3} we introduce the class of left--invariant functionals.
Finally, Section \ref{section 4} is mainly devoted to the proof of the main results, i.e., Theorem \ref{Th1} 
and Corollary \ref{Th2}.

\section{Preliminaries}\label{Prel}

We start this section recalling the basic notions of Carnot groups.

A Carnot group $\mathbb{G}= (\R^n,\cdot)$ is a connected, simply connected 
and nilpotent Lie group, whose Lie algebra $\mathfrak{g}$ admits a stratification, 
namely 
there exist linear subspaces, usually called {\it layers}, such that
$$\mathfrak{g}=V_1\oplus..\oplus V_k, \quad [V_1,V_i]=V_{i+1}, \quad V_{k}\neq \{0\}, \quad V_i=\{0\} \, \textrm{if $i>k$},$$
\noindent where $k$ is usually called the {\it step} of the group $(\mathbb{G},\cdot)$ and
$$[V_i,V_j]:=\textrm{span}\left\{[X,Y]: \, X\in V_i,Y\in V_j\right\}.$$

The explicit expression of the group law $\cdot$ can be deduced from the Hausdorff-Campbell formula, see e.g. \cite{BLU}.
The group law can be used to define a diffeomorphism, usually called {\it left--translation} 
$\gamma_y : \G \to \G$ for every $y\in \G$, defined as
$$\gamma_y (x) := y \cdot x \quad \textrm{for every } x \in \G.$$
A Carnot group $\G$ is also endowed with a family of automorphisms of the group $\delta_\lambda:\G\to\G$, $\lambda\in\mathbb{R}^+$, called {\it dilations},
given by
\begin{equation*}
   \delta_\lambda(x_1,\ldots,x_n):=(\lambda^{d_1}x_1,..,\lambda^{d_n}x_n),
\end{equation*}
\noindent where $(x_1,\ldots, x_n)$ are the {\it exponential coordinates} of $x \in \G$,
$d_{j}\in \N$ for every $j=1,\ldots,n$ and $1=d_1=\ldots=d_m < d_{m+1} \leq \ldots \leq d_{n}$
for $m := {\rm dim}(V_1)$. Here the group $\G$ and the algebra $\mathfrak{g}$ are identified through the exponential mapping.
The $n$-dimensional Lebesgue measure $\mathcal{L}^{n}$ of $\mathbb{R}^n$ 
provides the Haar measure on $\G$, see e.g. \cite[Proposition 1.3.21]{BLU}.


It is customary to denote with $Q:=\sum_{i=1}^k i\ {\rm dim}(V_i)$ the homogeneous dimension of $\mathbb{G}$ 
which corresponds to the Hausdorff dimension of $\G$ (w.r.t. an appropriate sub--Riemannian distance, see below). 
This is generally greater than or equal to the topological dimension of $\G$ and it coincides with it
only when $\G$ is the Euclidean group $(\R^n,+)$, which is the only Abelian Carnot group.\\

Carnot groups are also naturally endowed with sub-Riemannian distances which make them interesting examples of metric spaces. 
A first well--known example of such metrics is provided by the Carnot-Carath\'eodory distance $d_{cc}$, see e.g. \cite[Definition 5.2.2]{BLU}, which
is a path--metric resembling the classical Riemannian distance.
In our case, we will work with metrics induced by homogeneous norms.
\begin{definition}
A homogeneous norm $|\cdot|_\G:\G\to\mathbb{R}^+_0$ is a continuous function with the following properties:
\begin{itemize}
    \item [$(i)$] $|x|_\G=0\ \textrm{if and only if } x=0$ for every $x\in\mathbb{G}$;
    \item [$(ii)$] $|x^{-1}|_\G=|x|_\G$ for every $x\in\mathbb{G}$;
    \item [$(iii)$] $|\delta_\lambda x|_\G=\lambda |x|_\G$ for every $\lambda\in\mathbb{R}^+$ and for every $x\in\mathbb{G}$.
\end{itemize}
\end{definition}
A homogeneous norm induces a left--invariant homogeneous distance by
\begin{equation*}
    d(x,y):=|y^{-1}\cdot x|_\G\quad\textrm{for every }x,y\in\G.
\end{equation*}
We remind that a generic distance $d$ is left--invariant if and only if $d(z\cdot x,z\cdot y)=d(x,y)$ for every $x,y,z\in\mathbb{G}$.
A concrete example of such kind of homogeneous distance is given by the Kor\'anyi distance, see e.g. \cite{Cygan}.\\

%
For our purposes, we are also interested in introducing a right--invariant distance
$d^{\mathcal{R}}$ given by
$$d^{\mathcal{R}}(x,y) := |x \cdot y^{-1}|_\G \quad\textrm{for every } x,y \in \mathbb{G}.$$
As before, $d^\mathcal{R}$ is right--invariant if and only if $d^\mathcal{R}(x\cdot z,y\cdot z)=d^\mathcal{R}(x,y)$ for every $x,y,z\in\mathbb{G}$.

From now on we will write $B(x,\eps)$ and $B^{\mathcal{R}}(x,\eps)$ 
to denote the balls of center $x\in\mathbb{G}$ and radius $\eps>0$ 
w.r.t the distances $d$ and $d^{\mathcal{R}}$ respectively.
We notice that for any $\eps>0$ $$B(0,\eps)=B^{\mathcal{R}}(0,\eps).$$
\medskip
We also define two {\it left--translation operators},
one acting on functions and the other one acting on sets, which will be relevant 
in the upcoming sections. 
\begin{definition}\label{TranslationOp}
Let $y \in \G$ be any point. We define $\tau_{y}:L_{\mathrm{loc}}^p(\G)\to L_{\mathrm{loc}}^p(\G)$ as
$$\tau_{y}u(x):=u(y^{-1}\cdot x) \quad \textrm{for every } x \in \G.$$
With an abuse of notation, we also define $\tau_{y}:\mathcal{A}_0 \to\mathcal{A}_0$ as
$$\tau_{y}A:=y\cdot A =\{x \in \G :y^{-1}\cdot x\in A\},$$
where $\mathcal{A}_0$ denotes the family of all
bounded open sets of $\G$.
\end{definition}

We now want to introduce the relevant Sobolev spaces needed in the rest of the paper. 
Let $u:\G \to \R$ be a sufficiently smooth function and let $(X_1,..,X_m)$ be a basis of the horizontal layer $V_1$, 
made of left-invariant vector fields, i.e., $X_j(\tau_y u)=\tau_y(X_j u)$ for any $j=1,\dots,m$ and for any $y\in\G$. Then the {\it horizontal gradient} $\nabla_{\G}u$ of $u:\G \to \R$ is given by
$$\nabla_\mathbb{G}u:=\sum_{j=1}^m(X_j u)X_j=(X_1u,..,X_m u).$$

\begin{definition}\label{DefSobolevSpace}
Let $\Omega \subset \G$ be an open set and let $1\leq p<+\infty$. 
The Sobolev space $W^{1,p}_{\mathbb{G}}(\Omega)$ is given by
$$W^{1,p}_{\mathbb{G}}(\Omega):=\left\{u\in L^p(\Omega): \nabla_\mathbb{G}u\in (L^ p(\Omega))^m\right\}.$$
Moreover,
$$W^{1,p}_{\mathbb{G},\mathrm{loc}}(\Omega) := \left\{ u \in W^{1,p}_{\mathbb{G}}(\Omega'), \textrm{for every open set } \Omega'\Subset \Omega \right\}.$$
\end{definition} 


\medskip

The next part of this section is devoted to the introduction and a brief recap of
the main properties of the {\it local convolution}
recently introduced in \cite{CM}. First, we need to recall
the notion of {\it smooth mollifier}.

\begin{definition}\label{def_moll}
Given a smooth compactly supported function $\varphi \in C^{\infty}_{0}(B(0,1))$, for $\eps >0$ 
we define
the family of functions
$\varphi_{\eps}:\mathbb{G} \to \R$ as
$$\varphi_{\eps}(x) := \tfrac{1}{\eps^{Q}} \varphi \left( \delta_{\eps^{-1}}x\right).$$
We say that $\{\varphi_{\eps}\}_\eps$ is a family of mollifiers if it satisfies the
following conditions:
\begin{itemize}
\item $\varphi_{\eps}\geq 0$ in $\mathbb{G}$, for all $\eps>0$;
\item $\mathrm{supp}(\varphi_{\eps})\subset B(0,\eps)$, for all $\eps>0$;
\item $\int_{B(0,\eps)} \varphi_{\eps}\,dx=1$, for all $\eps>0$.
\end{itemize}
\end{definition}

Following \cite{CM}, we move to a proper definition of local convolution.
Let $\Omega \subset \mathbb{G}$ be any open set. 
For every $\eps>0$ we can define
the open set
$$\Omega_{\eps}^{\mathcal{R}} := \left\{ x \in \mathbb{G}: \mathrm{dist}^{\mathcal{R}}(x,\G\setminus\Omega) > \eps \right\}$$
where
$$\mathrm{dist}^{\mathcal{R}}(x,\G \setminus \Omega)  := \inf \left\{ d^{\mathcal{R}}(x,y): y\in\G\setminus\Omega\right\}.$$
Let $\varphi$ be a smooth mollifier with support within $B(0,1)$. 
For any $u\in L^{1}_{\mathrm{loc}}(\Omega)$ and $x\in\mathbb{G}$, we can define the local convolution
\begin{equation*}
u_\eps(x):=(\varphi_{\eps} \ast u)(x) :=\int_{\Omega} \varphi_{\eps}(x\cdot y^{-1})u(y)\,dy.
\end{equation*}
If we restrict the domain of definition by considering $x\in\Omega_{\eps}^{\mathcal{R}}$, we can write
\begin{equation}\label{LocalConvolution}
\begin{aligned}
(\varphi_{\eps} \ast u)(x) &= \int_{B^{\mathcal{R}}(x,\eps)} \varphi_{\eps}(x\cdot y^{-1}) u(y)\,dy 
= \int_{B(0,\eps)}\varphi_{\eps}(y)u(y^{-1}\cdot x)\,dy\\
&= \int_{B(0,1)} \varphi(z) u\left( (\delta_{\eps}z)^{-1}\cdot x\right) \, dz.
\end{aligned}
\end{equation}
\noindent where we used that

for every $\eps>0$, $B^{\mathcal{R}}(0,\eps)=B(0,\eps)$. 
We are finally ready to state the natural counterparts of the classical 
results holding for the Euclidean convolution, see e.g. \cite{EG}.
We refer to \cite{FS} for the analogous
result when dealing with {\it global convolution} on $\G$.

\begin{proposition}\label{propconv}
Let $\Omega\subset\G$ be an open set and let $p\in [1,+\infty)$.
Let $u\in L^{p}_{\mathrm{loc}}(\Omega)$ and let $\{\varphi_{\eps}\}_\eps$
a family of mollifiers according to Definition \ref{def_moll}. 
Then
\begin{equation}\label{LplocConvergence}
\varphi_{\eps} \ast u \longrightarrow u \quad \textrm{(strongly) in } L^{p}_{\mathrm{loc}}(\Omega).
\end{equation}
Moreover, if $u\in W^{1,p}_{\G,\mathrm{loc}}(\Omega)$, then
\begin{equation}\label{SobLocConvergence}
\varphi_{\eps} \ast u \longrightarrow u \quad \textrm{(strongly) in } W^{1,p}_{\G,\mathrm{loc}}(\Omega).
\end{equation}
\begin{proof}
The proof of \eqref{LplocConvergence} follows from similar arguments of the classical Euclidean proof, see e.g. \cite[Theorem 4.1]{EG}.
We report it here for the sake of completeness.

Let $u\in L^p_{\mathrm{loc}}(\Omega)$ and let us pick a point $x\in V\Subset W\Subset\Omega$, with $V,W$ being open sets.
Since for $\varepsilon>0$ small enough $V\subset \Omega_{\eps}^{\mathcal{R}}$, we can exploit \eqref{LocalConvolution}.
We first prove an auxiliary estimate which holds true for $p \in (1,+\infty)$. In this case, let us
set $p'$ to be conjugate exponent of $p$, namely $\tfrac{1}{p} + \tfrac{1}{p'} =1$. We find
\begin{equation*}
\begin{aligned}
|u_\varepsilon (x)|&\leq \int_{B(0,1)}\varphi(z) \left|u\left( (\delta_{\eps}z)^{-1}\cdot x\right)\right| \, dz
= \int_{B(0,1)}\varphi(z)^{\tfrac{1}{p}} \varphi(z)^{\tfrac{1}{p'}} \left|u\left( (\delta_{\eps}z)^{-1}\cdot x\right)\right| \, dz\\
		&\leq \left(\int_{B(0,1)}\varphi(z)\, dx \right)^{\tfrac{1}{p'}} \left( \int_{B(0,1)} \varphi(z) \left|u\left( (\delta_{\eps}z)^{-1}\cdot x\right)\right|^{p} \, dz\right)^{\tfrac{1}{p}}\\
		&=\left( \int_{B(0,1)} \varphi(z) \left|u\left( (\delta_{\eps}z)^{-1}\cdot x\right)\right|^{p} \, dz\right)^{\tfrac{1}{p}}.
\end{aligned}
\end{equation*}		
Hence, and now for every $p\in [1,+\infty)$, we obtain that
\begin{align*}
\|u_\varepsilon\|_{L^p(V)}^p &\leq 
\int_{B(0,1)}\varphi(z) \left( \int_{V}\left|u\left( (\delta_{\eps}z)^{-1}\cdot x\right)\right|^{p}\,dx \right)\, dz\\
&\leq \int_{W}|u(y)|^{p}\, dy =  \|u\|_{L^p(W)}^p
\end{align*}
for $\eps >0$ sufficiently small.\\

\noindent Let us now fix $\delta>0$. Since $u\in L^p(W)$, there exists $v\in C(\overline{W})$ such that $\|u-v\|_{L^p(W)}\leq\delta$. Moreover, by the last estimate, $\|u_\varepsilon-v_\varepsilon\|_{L^p(V)}\leq\|u-v\|_{L^p(W)}\leq\delta$. Thus
\begin{align*}
    \|u_\varepsilon-u\|_{L^p(V)}\leq\|u_\varepsilon-v_\varepsilon\|_{L^p(V)}+\|v_\varepsilon-v\|_{L^p(V)}+\|v-u\|_{L^p(V)}\leq 3\delta
\end{align*}
since $v_\varepsilon$ converges uniformly to $v$ in any compact subset of $W$.

Let us now move to the proof of \eqref{SobLocConvergence}. Thanks to \eqref{LplocConvergence},
it is enough to prove that
$$X_j u_{\eps} = \varphi_{\eps} \ast X_j u \quad \textrm{in } \Omega_\varepsilon^\mathcal{R}$$
\noindent for every $j=1,\ldots,m$. 
Let us fix $x\in\Omega_\varepsilon^\mathcal{R}$. By the left-invariance of each vector field $X_j$, we get
\begin{align*}
X_j u_{\eps}(x)&=X_j \left(\int_{B(0,\eps)}\varphi_{\eps}(y)u(y^{-1}\cdot z)\, dy \right)\big|_{z=x}
=\int_{B(0,\eps)}X_j(\varphi_{\eps}(y)u(y^{-1}\cdot x))\, dy\\
&=\int_{B(0,\eps)}\varphi_\varepsilon(y)(X_j u)(y^{-1}\cdot x)\, dy=(\varphi_{\eps}\ast X_{j}u)(x)
\end{align*}
as desired.
\end{proof}
\end{proposition}

\medskip
We close this section recalling the following version of the Jensen Inequality in Banach spaces.
\begin{lemma}\label{Jensen Inequality}
Let $X$ be a Banach space and let $F:X\rightarrow[0,+\infty]$ be a lower semicontinuous convex function. 
Let $(E,\epsilon,\mu)$ be a measure space with $\mu\geq 0$ and $\mu(E)=1$. Then,
\begin{align}\label{Jensen}
F\left(\int_E u(s)\,d\mu(s)\right)\leq \int_E F(u(s))\,d\mu(s)
\end{align}
for every $\mu$--integrable function $u:E\to X$.
\end{lemma}
We refer to \cite[Lemma 23.2]{DM} for a proof.

\section{Left-invariant functionals}\label{section3}
In this section we introduce the object of our study.
We recall that $\mathcal{A}_0$ denotes the family of all
bounded open sets of $\G$ and, from now on, we consider $p\in[1,\infty)$.
First of all, let us recall few definitions already appeared in the introduction.
\begin{definition}
Let $\alpha:\,\mathcal{A}_0\to [0,\infty]$ be a set function.
We say that:
\begin{itemize}
\item[(i)] $\alpha$ is {\it increasing} if $\alpha(A)\le\, \alpha(B)$, for each $A,\,B\in\mathcal{A}_0$ with $A\subseteq B$;
\item[(ii)] $\alpha$ is {\it inner regular } if
\[
\alpha(A)=\,\sup\left\{\alpha(B):\,B\in\mathcal{A}_0,\,B\Subset A\right\}\,\text{for each}\ A\in\mathcal{A}_0;
\]
\item[(iii)] $\alpha$ is {\it subadditive }if $\alpha(A)\le\,\alpha(A_1)+\alpha(A_2)$ for every $A,\,A_1,\,A_2\in\mathcal{A}_0$ with $A\subset A_1\cup A_2$;
\item[(iv)] $\alpha$ is {\it superadditive }if $\alpha(A)\ge\,\alpha(A_1)+\alpha(A_2)$ for every $A,\,A_1,\,A_2\in\mathcal{A}_0$ with $A_1\cup A_2\subseteq A$ and $A_1\cap A_2=\emptyset$;
\item[(v)] $\alpha$ is a {\it measure }if there exists a Borel measure $\mu:\,\mathcal B(\G)\to [0,\infty]$ such that $\alpha(A)=\,\mu(A)$ for every $A\in\mathcal{A}_0$.
\end{itemize}
\end{definition}
\begin{remark}
Let us recall that, if $\alpha:\,\mathcal{A}_0\to [0,\infty]$ is an increasing set function, then it is a measure if and only if it is subadditive, superadditive and inner regular. For details see, for instance, \cite[Theorem 14.23]{DM}.
\end{remark}
\begin{definition}\label{set functions}
Let $F:\,L^p_{\mathrm{loc}}(\G)\times\mathcal{A}_0\to [0,+\infty]$ be a functional. We say that:
\begin{itemize}
	\item $F$ is {\it increasing} if, for every $u\in L^p(\G)$, $F(u,\cdot):\,\mathcal{A}_0\to [0,+\infty]$ is increasing as set function;
	\item $F$ is {\it inner regular} (on $\mathcal{A}_0$) if it is increasing and, for each $u\in L^p(\G)$, $F(u,\cdot):\,\mathcal{A}_0\to [0,+\infty]$ is inner regular as set function;
	\item $F$ is a {\it measure}, if for every $u\in L^p(\G)$, $F(u,\cdot):\,\mathcal{A}_0\to [0,+\infty]$ is a measure as set function ;
	\item $F$ is {\it local} if
	\[
	F(u,A)=\,F(v,A)\,
	\]
	for each $A\in\mathcal{A}_0$, $u,v\in L^p(\G)$ such that $u=v$ a.e. on $A$;
	\item $F$ is {\it lower semicontinuous} if, for every $A\in\mathcal{A}_0$, $F(\cdot,A):\,L^p(\G)\to [0,+\infty]$ is lower semicontinuous;
	\item $F$ is {\it convex} if, for every $A\in\mathcal{A}_0$, $F(\cdot,A):\,L^p(\G)\to [0,+\infty]$ is convex.
\end{itemize}
\end{definition}
\begin{remark}
Let $F:\,L^p_{\mathrm{loc}}(\G)\times\mathcal{A}_0\to [0,+\infty]$ be a non-negative increasing functional such that $F(u,\emptyset)=0$ for every $u\in L^p_{\mathrm{loc}}(\G)$. Then, by \cite[Theorem 14.23]{DM}, $F$ is a measure if and only if $F$ is subadditive, superadditive and inner regular.
\end{remark}
By means of the operators introduced in Definition \ref{TranslationOp}, we are ready to define
the class of {\it left--invariant functionals}. 

\begin{definition}\label{DefLIF}
Let $F:L^p_{\mathrm{loc}}(\G)\times\mathcal{A}_0\to\overline{\mathbb{R}}$
be a functional. 
$F$ is left--invariant if for every $y \in \G$
\begin{equation}\label{LIC}
F(\tau_{y}u,\tau_{y} A)=F(u,A)
\end{equation}
\noindent for every $u \in L^p_{\mathrm{loc}}(\G)$ and for every $A\in \mathcal{A}_0$.
\end{definition}

We stress that whenever $\G = (\R^{n},+)$, the above definition
boils down to the one considered in \cite[Chapter 23]{DM} and it is
therefore possible to provide many examples of translation--invariant functionals.\\
A less trivial example directly adapted to the Carnot group situation is provided
by the following functional
\begin{equation}\label{LIT}
F(u,A):=\left\{\begin{array}{rl}
\int_A f(\nabla_{\G} u(x))\, dx & \textrm{if } u\in W^{1,1}_{\G,\mathrm{loc}}(A),\\
+\infty & \textrm{otherwise},
\end{array}\right.
\end{equation}
\noindent where $f$ is a non-negative Borel function.
We remind that the functional $F$ defined above is increasing, subadditive, inner regular and, therefore, it is a measure. For details, see e.g. \cite[Example 15.4]{DM}.
\begin{proposition}\label{prop 3}
Let $F:\,L^p_{\mathrm{loc}}(\G)\times\mathcal{A}_0\to \overline{\mathbb{R}}$
be a functional as in \eqref{LIT}. Then, $F$ is left--invariant.
\end{proposition}
\begin{proof}
First, we notice that, due to the left-invariance of the vector fields $X_1,\dots,X_m$, for any $u\in L_{\mathrm{loc}}^p(\G)$ and $A\in\mathcal{A}_0$
$$\tau_{y}u\in W_{\G,\mathrm{loc}}^{1,1}(\tau_yA) \ \textrm{if and only if } u\in W_{\G,\mathrm{loc}}^{1,1}(A).$$
Therefore, it is sufficient to prove the result for functions 
$u\in L_{\mathrm{loc}}^p(\G)$ such that $u|_A\in W_{\G,\mathrm{loc}}^{1,1}(A)$.

Let us fix $u\in L^p_{\mathrm{loc}}(\G)$ and $A\in\mathcal{A}_0$ such that $u|_A\in W^{1,1}_{\G,\mathrm{loc}}(A)$. 
By a change of variables, it follows that
\begin{equation}\label{1.3}
F(\tau_yu,\tau_yA)=\int_{\tau_yA}f(\nabla_{\G}\tau_y u(x))\, dx=\int_A f(\nabla_{\G}u(z))\, dz=F(u,A)
\end{equation}
\noindent as desired.
\end{proof}
\begin{remark}\label{prop 3p}
The previous result trivially holds if we replace $W^{1,1}_{\G,\mathrm{loc}}(A)$ with $W^{1,p}_{\G,\mathrm{loc}}(A)$ for $p>1$.
\end{remark}

We close this section by proving a couple of auxiliary results needed in the upcoming section.
The first one is the natural counterpart of a classical result of Carbone and Sbordone, see \cite{CS}.

\begin{theorem}\label{DMThm23.1ext}
Let $F:\,L^p_{\mathrm{loc}}(\G)\times\mathcal{A}_0\to\overline{\mathbb{R}}$ be a left--invariant, increasing, convex and lower semicontinuous functional and let $\{\varphi_{h}\}_{h\in \mathbb{N}}$ be a sequence of mollifiers, as in Definition \ref{def_moll}. Then
\begin{align*}
F(u,A')\leq\liminf_{h\to+\infty}F(\varphi_h*u,A')\leq\limsup_{h\to+\infty}F(\varphi_h*u,A')\leq F(u,A)
\end{align*}
for every $u\in L^p_{\mathrm{loc}}(\mathbb{G})$ and for every $A,\ A'\in\mathcal{A}_0$ with $A'\Subset A$.

\begin{proof}
The first inequality trivially follows from the lower semicontinuity of $F$, while
the second is always trivially satisified.

It remains to prove that
\begin{equation}\label{23.1}
\limsup_{h\to+\infty}F(u_h,A')\leq F(u,A),
\end{equation}
where $u_h:=\varphi_h*u$. To this aim, let us fix $u\in L^p_{\mathrm{loc}}(\mathbb{G})$, $A,A'\in\mathcal{A}_0$ such that $A'\Subset A$. Moreover, let $h\in\mathbb{N}$ be such that $\frac{1}{h}<\mathrm{dist}^{\mathcal{R}}(A',\mathbb{G}\setminus A)$ 
and let us define $B_h:=B(0,\frac{1}{h})$. We can notice that, for every $x\in A'$
\begin{equation}\label{conv}
u_h(x)=\int_{B_h}u(y^{-1}\cdot x)\varphi_h(y)\, dy=\int_{B_h}\tau_yu(x)\varphi_h(y)\, dy
\end{equation}
By \eqref{conv}, Lemma \ref{Jensen Inequality} and being $F$ left--invariant, we get
\begin{equation*}
\begin{aligned}
F(u_h,A')&=F\left(\int_{B_h}\tau_yu\ \varphi_h(y)\, dy,A'\right)\leq\int_{B_h}F(\tau_yu,A')\varphi_h(y)\, dy\\
&=\int_{B_h}F(u,\tau_{y^{-1}}A')\varphi_h(y)\, dy \leq\int_{B_h}F(u,A)\varphi_h(y)\, dy=F(u,A)
\end{aligned}
\end{equation*}
where the last inequality follows observing that $\tau_{y^{-1}}A'\subset A$ for each $y\in B_h$. 
Indeed, for any $x\in\tau_{y^{-1}}A'$, that is $y\cdot x\in A'$, if $x\in\mathbb{G}\setminus A$ we would have
\begin{equation*}
d^\mathcal{R}(x,y\cdot x)=|y|_{\G}<\frac{1}{h}<d^\mathcal{R}(A',\mathbb{G}\setminus A),
\end{equation*}
\noindent which is impossible.
Then, taking the $\limsup$ as $h \to +\infty$ we get \eqref{23.1}.
\end{proof}
\end{theorem}

The next result yields the lower semicontinuity of integral functionals of the form \eqref{LIT},
under appropriate assumptions on the integrand. See \cite{Serrin} for the Euclidean case.

\begin{theorem}\label{DMThm23.3ext}
Let $f:\mathbb{R}^m\rightarrow[0,+\infty]$ be a convex and lower semicontinuous function and let $A$ be an open subset of $\mathbb{G}$. 
Then, the functional $F:\,W^{1,1}_{\G,\mathrm{loc}}(A)\to\mathbb{R}$, defined as
$$F(u):=\int_A f(\nabla_{\G}u(x))\,dx$$
is lower semicontinuous on $W_{\G,\mathrm{loc}}^{1,1}(A)$ with respect to the topology induced by $L^1_{\mathrm{loc}}(A)$.
\end{theorem}
\begin{proof}
Let us fix $A$ open subset of $\mathbb{G}$ and $u_h,u\in W^{1,1}_{\G,\mathrm{loc}}(A)$ 
such that $u_h\to u$ in $L^1_{\mathrm{loc}}(A)$. We just need to show that
\begin{equation}\label{lsc}
F(u)
\leq\liminf_{h\to+\infty}\int_A f(\nabla_{\G}u_h)\,dx.
\end{equation}
To this aim, let us fix $A'\Subset A$, $k\in\mathbb{N}$ such that $\frac{1}{k}<\mathrm{dist}^\mathcal{R}(A',\mathbb{G}\setminus A)$ and let us consider a sequence of mollifiers $\{\varphi_k\}_{k \in \N}$ as in Definition \ref{def_moll}. Moreover, let us denote $B_k:=B\left(0,\frac{1}{k}\right)$.
By Lemma \ref{Jensen Inequality} and Proposition \ref{prop 3}, we have
\begin{equation}\label{stima}
\begin{aligned}
\int_{A'}f(\nabla_{\G}(\varphi_k*u_h)(x))\,dx
&=\int_{A'} f((\varphi_k*\nabla_{\G}u_h)(x))\,dx\\
&=\int_{A'} f\left(\int_{B_k}\nabla_{\G}u_h(y^{-1}\cdot x)\varphi_k(y)\,dy\right)dx\\
&\leq\int_{A'}\left(\int_{B_k} f(\nabla_{\G}u_h(y^{-1}\cdot x))\varphi_k(y)\,dy\right)\,dx\\
&=\int_{B_k}\left(\int_{A'}f(\nabla_{\G}\tau_y u_h(x))\,dx\right)\varphi_k(y)\,dy\\ 
&=\int_{B_k}\left(\int_{\tau_{y^{-1}}A'}f(\nabla_{\G}u_h(x))\,dx\right)\varphi_k(y)\,dy\\
&\leq\int_{B_k}\left(\int_{A}f(\nabla_{\G}u_h(x))\,dx\right)\varphi_k(y)\,dy
=\int_{A}f(\nabla_{\G}u_h(x))\,dx
\end{aligned}
\end{equation}
\noindent where the last inequality follows from the same arguments used in the proof of Theorem \ref{DMThm23.1ext}.

Let us now show that
\begin{equation}\label{conv_infinito}
\varphi_k*u_h\to\varphi_k*u\ \text{in }C^\infty(\overline{A'}).
\end{equation}
Recalling that $u_h\to u\ \text{in }L^1_{\mathrm{loc}}(A)$ as $h\to +\infty$ by Proposition \ref{propconv} then, for each $\alpha,h\in\mathbb{N}$ and for every $x\in\overline{A'}$ and $j=1,..,m$, it holds that
\begin{align*}
\left|(X_j^\alpha(\varphi_k*u_h)-X_j^\alpha(\varphi_k*u))(x)\right|
&=\left|X_j^\alpha(\varphi_k*u_h-\varphi_k*u)(x)\right|\\
&=\left|X_j^\alpha\left(\int_{B^{\mathcal{R}}\left(x,\frac{1}{k}\right)}(u_h(y)-u(y))\varphi_k(x\cdot y^{-1})\,dy)\right)\right|\\
&=\left|\int_{B^{\mathcal{R}}\left(x,\frac{1}{k}\right)}(u_h(y)-u(y))X_j^\alpha\varphi_k(x\cdot y^{-1})\,dy)\right|\\
&\leq\int_{B^{\mathcal{R}}\left(x,\frac{1}{k}\right)}\left|u_h(y)-u(y)\right|\left|X_j^\alpha\varphi_k(x\cdot y^{-1})\right|\,dy\\
&\leq\|X_j^\alpha\varphi_k\|_{L^\infty(A)}\int_{A}\left|u_h(y)-u(y)\right|\,dy.
\end{align*}
Passing to the supremum in $\overline{A'}$ and taking the limit as $h\to +\infty$, we get \eqref{conv_infinito}.
As a consequence, the sequence $\{\nabla_{\G}(\varphi_k*u_h)\}_h$ uniformly converges to $\nabla_{\G}(\varphi_k*u)$ in $A'$, as $h\to+\infty$.

We can also notice that, by the lower semicontinuity of $f$, by \eqref{stima} and applying the Fatou's Lemma, then
\begin{equation}\label{Fatou}
\begin{aligned}
\int_{A'}f(\nabla_{\G}(\varphi_k*u))\,dx\leq\liminf_{h\to+\infty}\int_{A'}f(\nabla_{\G}(\varphi_k*u_h))\,dx\leq\liminf_{h\to+\infty}\int_{A}f(\nabla_{\G}u_h)\,dx.
\end{aligned}
\end{equation}
Moreover, being $\nabla_{\G}(\varphi_k*u)$ convergent to $\nabla_{\G}u$ in $L^1(A')$, in according with Proposition \ref{propconv}, we finally get, by the lower semicontinuity of $f$, the Fatou's Lemma and by \eqref{Fatou}
\begin{align*}
\int_{A'}f(\nabla_{\G}u)\,dx\leq\liminf_{k\to+\infty}\int_{A'}f(\nabla_{\G}(\varphi_k*u))\,dx\leq\liminf_{h\to+\infty}\int_{A}f(\nabla_{\G}u_h)\,dx.
\end{align*}
\end{proof}

We close this section recalling a definition which will be useful in the sequel. See, for instance, \cite[Chapter 15]{DM} for details.
\begin{definition}\label{envelope}
Let $X$ be a topological space and let $F:X\times\mathcal{A}_0\to\overline{\mathbb{R}}$ be an increasing functional, in according with Definition \ref{set functions}. We define the \it{inner regular envelope} of $F$ the increasing functional $F_{-}:X\times\mathcal{A}_0\to\overline{\mathbb{R}}$ defined as
\begin{align*}
F_{-}(x,A):=\sup\{F(x,B):B\in\mathcal{A}_0, B\Subset A\}
\end{align*}
for every $x\in X$ and for every $A\in\mathcal{A}_0$.

Moreover, we define the \it{lower semicontinuous envelope} of $F$ the functional $sc^{-}F:X\times\mathcal{A}_0\to\overline{\mathbb{R}}$, defined as
\begin{align*}
(sc^{-}F)(x,A):=\sup_{U\in\mathcal{N}(x)}\inf_{y\in U}F(y,A)
\end{align*}
for every $x\in X$ and for every $A\in\mathcal{A}_0$, where $\mathcal{N}(x)$ denotes the set of all open neighbourhoods of $x$ in $X$.
\end{definition}
\begin{remark}
If the functional $F$ is increasing and lower semicontinuous, then $F_{-}$ is also increasing, lower semicontinuous and inner regular. If $F$ is just increasing, then $sc^{-}F$ is still increasing and lower semicontinuous, but, in general, it is not inner regular, even if $F$ is inner regular. See for instance \cite[Example 15.11]{DM}.

Finally, named $\overline{F}:=(sc^{-}F)_{-}$ the inner regular envelope of the lower semicontinuous envelope of $F$, then $\overline{F}$ is the greatest increasing, inner regular and lower semicontinuous functional less than or equal to $F$.
\end{remark}


\section{Proof of Theorem \ref{Th1}}\label{section 4}

In this section we can finally prove the main results of the paper.

\begin{proof}[Proof of Theorem \ref{Th1}]
We start defining the auxiliary function $u_{\xi}:\G \to \R$ as
\begin{equation}\label{uxi}
u_{\xi}(x):= \scal{\xi}{\Pi(x)}_{\R^{m}} \quad \textrm{for every } \xi \in \R^{m},
\end{equation}
\noindent where $\Pi: \R^n \to \R^m$ denotes the projection over the horizontal layer $V_1$, here identified with
$\R^m$.\\
We note that $u_{\xi}$ is smooth and
\begin{equation}\label{GradUxi}
\nabla_{\G}u_{\xi}(x)= \xi \quad \textrm{for every } x \in \G.
\end{equation}
We also note that
\begin{equation}\label{ContoFacile}
\tau_{x}u_{\xi}(y) = u_{\xi}(x^{-1}\cdot y) = \scal{\xi}{\Pi(x^{-1}\cdot y)}_{\R^{m}} 
= \sum_{i=1}^{m}\xi_{i}(y_{i}- x_{i}) = u_{\xi}(y) - u_{\xi}(x)
\end{equation}
\noindent for every $x,y \in \G$.\\
Therefore, by the left-invariance of $F$, \eqref{ContoFacile} and by property $(d)$, we get 
\begin{equation}\label{InvarianceOfF}
\begin{aligned}
F(u_{\xi},B_{\rho}(0))&=F(u_{\xi}(y),B_{\rho}(0)) = F(\tau_{x}u_{\xi}(y),\tau_{x}B_{\rho}(0))\\
&=F(u_{\xi}(y)-c, B_{\rho}(x))=F(u_{\xi}, B_{\rho}(x))
\end{aligned}
\end{equation}
\noindent for every $x \in \G$. We stress that $c := u_{\xi}(x)$ is a constant w.r.t. $y$.

Now, by \cite[Theorem 3.12]{MPSC}, there exists a function $f:\G \times \R^m \to [0,+\infty]$ such that
$$F(u_{\xi},B_{\rho}(x)) = \int_{B_{\rho}(x)}f(z, \nabla_{\G}u_{\xi}(z))\, dz.$$
Since $|B_{\rho}(x)| = |B_{\rho}(0)|$ for every $x \in \G$, by \eqref{InvarianceOfF} and
Lebesgue's differentation Theorem, we have that, taking the limit as $\rho \to 0^{+}$,
$$f(0,\xi) \leftarrow \dfrac{1}{|B_{\rho}(0)|}\int_{B_{\rho}(0)}f(z, \nabla_{\G}u_{\xi}(z))\, dz =
\dfrac{1}{|B_{\rho}(x)|}\int_{B_{\rho}(x)}f(z, \nabla_{\G}u_{\xi}(z))\, dz \rightarrow f(x,\xi)$$
\noindent for every $x \in \G$.
Therefore, we can consider the well--defined function $f_0: \R^{m} \to [0,+\infty]$ given by
$$f_0(\xi):= f(0,\xi)\quad \textrm{for every } \xi \in \R^m.$$
Moreover, such function $f_0$ inherits all the properties of $f$ proved to hold in \cite[Theorem 3.12]{MPSC}.
Namely, $f_0$ is convex and 
$$0\leq f_0(\xi) \leq a + b |\xi|^{p} \quad \textrm{for every } \xi \in \R^m.$$
Moreover, 
\begin{equation}\label{RappresentazioneLisce}
F(u,A) = \int_{A} f_0(\nabla_{\G}u(x)) \, dx
\end{equation}
\noindent for every $u \in C^{\infty}(\G)$ and for every $A\in \mathcal{A}_0$.

It remains to show that the same representation \eqref{RappresentazioneLisce} holds for every
$u \in L^{p}_{\mathrm{loc}}(\G)$ with $u|_{A} \in W^{1,1}_{\G,\mathrm{loc}}(A)$ (for every $A \in \mathcal{A}_0$).\\
Let $A' \in \mathcal{A}_{0}$ such that $A' \Subset A$ and let $\{\varphi_{h}\}_h$ is a family of smooth mollifiers as in Definition \ref{def_moll} (here with $h=\tfrac{1}{\varepsilon}$). 
Hence, by \eqref{SobLocConvergence}, the Fatou's Lemma, the representation among smooth functions \eqref{RappresentazioneLisce}
and by Theorem \ref{DMThm23.1ext}, we get
\begin{equation*}
\begin{aligned}
\int_{A'}f_0(\nabla_{\G}u)\,dx  
\leq \liminf_{h\to +\infty} \int_{A'}f_{0}(\nabla_{\G}u_h)\, dx
=\liminf_{h\to+\infty} F(u_h, A') \leq F(u,A),
\end{aligned}
\end{equation*}
where $u_h:=\varphi_h \ast u$ for every $h\in\mathbb{N}$.\\
Therefore, by taking the supremum for $A'\Subset A$, we get
\begin{equation}\label{Ineq1}
\int_{A}f_0(\nabla_{\G}u(x))\,dx \leq F(u,A).
\end{equation}

We now proceed with the proof of the opposite inequality. First, we notice that, by Theorem \ref{DMThm23.3ext},
the functional $F$ is lower semicontinuous in $W^{1,1}_{\G,\mathrm{loc}}(A)$. Hence
\begin{equation}\label{Flsc}
F(u,A') \leq \liminf_{h\to +\infty}F(u_h, A').
\end{equation}
Now, as before, we denote by $B_{h}:= B\left(0, \tfrac{1}{h}\right)$ for every $h \in \mathbb{N}$.
Whenever $\tfrac{1}{h} < \mathrm{dist}^{\mathcal{R}}(A', \G \setminus A)$ then, by Lemma \ref{Jensen Inequality} and the left-invariance of $F$, it holds that
\begin{equation*}
\begin{aligned}
F(u_h, A') & = \int_{A'} f_{0} \left( \int_{B_h} \nabla_{\G}u(y^{-1}\cdot x)\varphi_{h}(y) \, dy \right) \, dx\\
&\leq \int_{A'}\left( \int_{B_h} f_{0}(\nabla_{\G}u(y^{-1}\cdot x))\varphi_{h}(y) \, dy \right) \, dx\\
&=\int_{B_h}\left( \int_{A'} f_{0}(\nabla_{\G}u(y^{-1}\cdot x))\, dx\right)\varphi_{h}(y) \, dy \\
&\leq\int_{B_h}\left( \int_{A} f_{0}(\nabla_{\G}u(x))\, dx\right)\varphi_{h}(y) \, dy \\
&=\int_{A}f_{0}(\nabla_{\G}u(x))\, dx.
\end{aligned}
\end{equation*}
We stress that the last inequality follows from the same argument used at the end of the proof of Theorem \ref{DMThm23.1ext}.\\
Combined with \eqref{Flsc}, this yields
$$F(u,A') \leq \int_{A}f_{0}(\nabla_{\G}u(x))\, dx,$$
\noindent which in turn gives
\begin{equation}\label{Ineq2}
F(u,A) \leq \int_{A}f_{0}(\nabla_{\G}u(x))\, dx,
\end{equation}
\noindent by passing to the supremum for $A'\Subset A$.
Taking into account \eqref{Ineq1} and \eqref{Ineq2}, we close the proof.
\end{proof}

\medskip

\indent As for the classical case, we can prove that left--invariant functionals are
uniquely determined on $L^{p}_{\mathrm{loc}}(\G)$ by their prescription on
a class of regular functions. Taking into account Definition \ref{envelope}, we preliminary need the following

\begin{theorem}\label{DM23.5}
Let $F:L^{p}_{\mathrm{loc}}(\G) \times \mathcal{A}_{0} \to [0,+\infty]$ be an increasing
functional satisfying the assumptions $(a)-(e)$ of Theorem \ref{Th1} and let $f:\R^{m} \to [0,+\infty]$
be as in Theorem \ref{Th1}. 
Let $\mathcal{F}:L^{p}_{\mathrm{loc}}(\G) \times \mathcal{A}_{0} \to [0,+\infty]$ be the
functional defined as
\begin{equation*}
\mathcal{F}(u,A):= \left\{\begin{array}{rl}
 \int_{A}f(\nabla_{\G}u(x))dx &\quad \textrm{if } u \in W^{1,1}_{\G,\mathrm{loc}}(A),\\							+\infty & \quad \textrm{otherwise}.					\end{array}\right.	
\end{equation*}
Let $\overline{\mathcal{F}}$ be the inner regular envelope of the lower semicontinuous envelope of $\mathcal{F}$.
Then,
\begin{equation}\label{Prima}
\overline{\mathcal{F}}(u,A) = \int_{A} f(\nabla_{\G}u(x))\, dx
\end{equation}
\noindent for every $A \in \mathcal{A}_{0}$ and for every $u \in L^{p}_{\mathrm{loc}}(\G)$ such that
$u|_{A} \in W^{1,1}_{\G,\mathrm{loc}}(A)$. \\
Moreover,
\begin{equation}\label{Seconda}
F(u,A) = \overline{\mathcal{F}}(u,A)
\end{equation}
\noindent for every $u\in L^{p}_{\mathrm{loc}}(\G)$ and for every $A \in \mathcal{A}_{0}$.
\begin{proof}
By Theorem \ref{DMThm23.3ext}, the functional $\mathcal{F}$ is lower semicontinuous on $W^{1,1}_{\G,\mathrm{loc}}(A)$ with the respect to the topology induced by $L^{1}_{\mathrm{loc}}(A)$.
Moreover, by Proposition \ref{prop 3}, $\mathcal{F}$ is also left--invariant. Finally, it is easy
to check that $\overline{\mathcal{F}}$ satisfies properties $(a)-(e)$ of Theorem \ref{Th1}. Therefore, 
\eqref{Prima} directly follows.

Concerning \eqref{Seconda}, we first recall that $\overline{\mathcal{F}}$ is an increasing, inner regular and
lower semicontinuous functional, which is also the greatest functional with these properties less
than or equal to $\mathcal{F}$.
Therefore, since $F(u,A) \leq \mathcal{F}(u,A)$ 
for every $A \in \mathcal{A}_0$ and for every $u \in L^{p}_{\mathrm{loc}}(\G)$, we get 
that
$$F(u,A) \leq \overline{\mathcal{F}}(u,A)$$
\noindent for every $A \in \mathcal{A}_0$ and for every $u \in L^{p}_{\mathrm{loc}}(\G)$.\\
\indent In order to complete the proof we need to show that the opposite inequality
holds true as well. To this aim, let us consider $u\in L^{p}_{\mathrm{loc}}(\G)$
and $A \in \mathcal{A}_0$. We then consider $A'\in\mathcal{A}_0$ such that $A' \Subset A$ and
a sequence of mollifiers $\{\varphi_{h}\}_{h\in \mathbb{N}}$ as in Definition \ref{def_moll}.\\
Since $u_h:=\varphi_{h} \ast u$ is smooth, then, by Theorem \ref{Th1} (see in particular \eqref{RappresentazioneLisce}),
we have
$$\mathcal{F}(u_h, A') = F(u_h, A')$$
\noindent for every $h\in \mathbb{N}$. Now, the lower semicontinuity of $\overline{\mathcal{F}}$, implies that
$$\overline{\mathcal{F}}(u,A') \leq \liminf_{h\to +\infty}\mathcal{F}(u_h, A')
\leq \limsup_{h\to +\infty} F(u_h, A') \leq F(u,A)$$
by Theorem \ref{DMThm23.1ext}.
Since $\overline{\mathcal{F}}$ is inner regular, then, taking the supremum among sets $A' \Subset A$, we finally get
$$\overline{\mathcal{F}}(u,A) \leq F(u,A)$$
\noindent for every $A \in \mathcal{A}_0$ and for every $u \in L^{p}_{\mathrm{loc}}(\G)$.

\end{proof}
\end{theorem}

As a direct consequence, we can finally prove 

\begin{theorem}
Let $F,G: L^{p}_{\mathrm{loc}}(\G) \times \mathcal{A}_{0} \to [0,+\infty]$ be two increasing functionals
satisfying $(a)-(e)$ of Theorem \ref{Th1}.
Let $\emptyset \neq A_0 \in \mathcal{A}_0$ and, for every $\xi \in \R^m$, let $u_{\xi}:\G \to \R$ be defined as in \eqref{uxi}. Moreover, assume that
\begin{equation}\label{FequalG}
F(u_{\xi}, A_0) = G(u_{\xi},A_0) \quad \textrm{for every } \xi \in \R^m.
\end{equation}
Then, $F=G$ on $L^{p}_{\mathrm{loc}}(\G) \times \mathcal{A}_{0}$.
\begin{proof}
By Theorem \ref{Th1}, we know the existence of two convex functions
$f,g : \R^m \to [0,+\infty]$ such that
$$F(u,A) = \int_{A}f(\nabla_{\G}u(x))\, dx \quad \textrm{and } \quad G(u,A) = \int_{A}g(\nabla_{\G}u(x))\, dx$$
\noindent for every $A \in \mathcal{A}_0$ and for every $u \in L^{p}_{\mathrm{loc}}(\G)$ such
that $u|_{A} \in W^{1,1}_{\G,\mathrm{loc}}(A)$.\\
Moreover, by \eqref{FequalG} and \eqref{GradUxi}, we get
$$f(\xi) |A_{0}| = F(u_{\xi},A_0) = G(u_{\xi},A_0) = g(\xi) |A_0|$$
\noindent for every $\xi \in \R^{m}$. Applying Theorem \ref{DM23.5} we get the desired conclusion. 
\end{proof}
\end{theorem}

We close the section with the 
\begin{proof}[Proof of Corollary \ref{Th2}]
By Theorem \ref{Th1}, there exists a sequence of convex functions $f_h:\mathbb{R}^m\to [0,+\infty]$ such that
$$F_h(u,A)=\int_A f_h(\nabla_{\G}u(x))\,dx$$
\noindent for every $h\in\mathbb{N}$ for every $A\in\mathcal{A}_0$ and 
for every $u\in L_{\mathrm{loc}}^p(\G)$ such that $u|_A\in W^{1,1}_{\G, \mathrm{loc}}(A)$.\\
Moreover, since the bounds on $F_h$ are uniform, we have that
	$$0\le f_h(\eta)\le a+\, b\, |\eta|^p$$
\noindent for every $\eta\in\mathbb{R}^m$ and for every $h\in\mathbb{N}$.

Now, by \cite[Theorem 4.20]{MPSC}, up to a subsequence, there exist 
a local functional $F:L^p_{\mathrm{loc}}(\G)\times\mathcal{A}_0 \to [0,+\infty]$ 
and a convex function $f:\mathbb{R}^m\to [0,+\infty]$ such that 
\begin{itemize}
\item[(i)]\begin{equation*}
0\leq f(\eta)\le a+b|\eta|^p \quad \textrm{for every } \eta\in\mathbb{R}^m;
\end{equation*}
\item[(ii)]
  \begin{equation*}
    F(\cdot,A)=\Gamma-\lim_{h\to+\infty} F_h(\cdot,A)\quad\textrm{for every }A\in\mathcal{A}_0\,;
    \end{equation*}
		\item[(iii)]
$F$ admits the following representation \begin{equation}\label{rep}
    F(u,A)=
    \displaystyle{\begin{cases}
    \int_{A}f(\nabla_\G u(x))dx&\text{ if }A\in\mathcal{A}_0,u|_{A}\in W^{1,p}_{\G,\mathrm{loc}} (A)\\
    +\infty&\text{ otherwise}
    \end{cases}\,.}
    \end{equation}
\end{itemize}
Finally, by Remark \ref{prop 3p} we can infer that the $\Gamma$-limit $F$ is
left--invariant.
\end{proof}
\begin{remark}
We remind that Corollary \ref{Th2} cannot be expected to hold in general even for $p=1$, since this is false already in the Euclidean case. We refer to \cite[Example 3.14]{DM} for more details.
\end{remark}

\section*{Acknowledgements}
The authors would like to thank the anonymous referee for the suggestions that improved the paper.
The authors would also like to thank Prof. Francesco Serra Cassano and Prof. Andrea Pinamonti of the Department of Mathematics of the University of Trento, for their precious support and help.

\end{document}